\documentclass[a4paper,11pt,reqno]{amsart}

\usepackage[utf8]{inputenc}
\usepackage[english]{babel}
\usepackage{amsmath}
\usepackage{amsfonts}
\usepackage{amsthm}
\usepackage{mathtools}
\usepackage{float}
\usepackage{stmaryrd}
\usepackage{graphics}
\usepackage{graphicx}
\usepackage{subfig}
\usepackage{enumitem}
\usepackage{crossreftools}
\usepackage{bm}
\usepackage[colorlinks=false]{hyperref}
\usepackage{color}

\newcommand{\mc}{\mathcal}
\newcommand{\mf}{\mathfrak}
\newcommand{\eps}{\varepsilon}
\newcommand{\epsn}{{\varepsilon_n}}
\renewcommand{\d}{\,\mathrm{d}}

\DeclareMathOperator*{\argmin}{argmin}

\newcommand\R{\mathbb R}
\newcommand\N{\mathbb N}

\renewcommand{\to}{\rightarrow}

\numberwithin{equation}{section}
\newtheorem{thm}{Theorem}[section]
\newtheorem{defn}[thm]{Definition}
\newtheorem{prop}[thm]{Proposition}
\newtheorem{lemma}[thm]{Lemma}
\newtheorem{cor}[thm]{Corollary}

\theoremstyle{definition}
\begingroup
\newtheorem{rmk}[thm]{Remark}
\newtheorem{ex}[thm]{Example}
\endgroup

\theoremstyle{remark}

\title[From cohesive to brittle quasistatic debonding]{From cohesive to brittle debonding: the quasistatic framework}

\author[F. Riva]{Filippo Riva}

\address[F. Riva]{Charles University, Faculty of Mathematics and Physics, Department of Mathematical Analysis, Sokolovsk\'a 49/83, 186 75 Prague 8, Czech Republic}
\email{filippo.riva@matfyz.cuni.cz}

\begin{document}

	\begin{abstract}		
		The approximation of brittle laws via steeper and steeper cohesive profiles is validated within the mechanical setting of debonding models, which describe the detachment process of a peeled elastic adhesive membrane. In a quasistatic framework, energetic solutions to a suitably rescaled cohesive debonding problem, formulated in terms of displacements, are proved to converge to a limit evolution of shapes solving its brittle counterpart. The proposed approach relies on an equivalent and recently introduced free-boundary reformulation of this latter model.
	\end{abstract}
	
	\maketitle
	
	{\small
		\par
	}
	
	\pagenumbering{arabic}
	
	\medskip
	
	\tableofcontents
	
	\section*{Introduction}	
	In the last decades, the variational modelling of failure mechanisms in elastic materials has generated an increasing interest within the mathematical community, initiated by \cite{FrancMar} in the framework of fracture mechanics, and directly extended to delamination or debonding models. In literature, such degradation phenomena are usually classified in the two cathegories of brittle and cohesive. The former case, whose energetic interpretation has been pioneered by Griffith \cite{Griffith}, describes an abrupt process occurring instantaneously and which thus divides the object into a completely sound and a completely collapsed region. On the contrary, as envisaged by Barenblatt \cite{Barenblatt} (see also \cite{Dugdale}), the latter case illustrates a more gradual behaviour somehow introducing mixed zones in which the material is just partially broken.
	
	This indicates that cohesive laws could be interpreted as a regularized version of brittle ones. Indeed, at a formal level, the dissipation density of a brittle model is described by the characteristic function $\chi_{\{\lambda>0\}}$ of the positivity set of a certain nonnegative scalar quantity $\lambda$ (hence, the sets $\{\lambda=0\}$ and $\{\lambda>0\}$ represent the sound and the collapsed regions, respectively), while cohesive densities are modelled by a nondecreasing continuous function $\Phi(\lambda)$ satisfying $\Phi(0)=0$ and $\lim\limits_{\lambda\to +\infty}\Phi(\lambda)=1$ (so, the set $\{0<\Phi(\lambda)<1\}$ represents the partially broken part of the material). Since clearly $\chi_{\{\lambda>0\}}=\lim\limits_{\eps\to 0}\Phi(\frac \lambda\eps)$, one naturally expects that cohesive models approximate brittle ones when the cohesive profile becomes steeper and steeper. Although this can be easily made rigorous in a static scenario, where the theory of $\Gamma$-convergence \cite{DMbook} directly applies due to the monotonicity of the map $\eps\mapsto\Phi(\frac\lambda\eps)$, the problem becomes way more challenging in the evolutive framework due to the emergence of an additional crucial feature: the irreversibility of failure progression. In this setting, we are in fact not aware of any result in this direction.
	
	In the current paper we propose to verify such approximation in the context of debonding models, focusing on the quasistatic regime. This means that inertial effects are neglected and the process evolves through states of equilibrium \cite{MielkRoubbook}.
	
	Debonding models describe the evolution of an adhesive elastic membrane glued to a planar rigid substrate and peeled away due to the application of an external prescribed displacement. If the adhesion law between the membrane and the substrate is brittle, the process is usually formulated in terms of an evolution of sets $t\mapsto A(t)$, representing the debonded region at time $t$. In the quasistatic framework, the governing rules of the phenomenon are given by a global stability condition and an energy-dissipation balance (energetic solutions):
	\begin{equation}\label{eq:ensol}
		\begin{cases}
			\displaystyle\mathcal E(t,A(t))\le \mathcal E(t,B)+\int_{B\setminus A(t)}\kappa\,dx,\qquad\text{for all }B\supseteq A(t),\\
			\displaystyle\mathcal E(t,A(t))+\int_{A(t)\setminus A(0)}\kappa\,dx=\mathcal E(0,A(0))+\int_0^t P(\tau)\,d\tau.
		\end{cases}
	\end{equation}
Above, $\mc E(t,A)$ denotes the elastic energy of the membrane which at time $t$ is debonded on $A$ (see \eqref{eq:E}), while the integral term $\int_{B\setminus A}\kappa\,dx$ accounts for the dissipation of energy occurring by debonding the membrane from a certain configuration $A$ to a larger one $B$. Indeed, the function $\kappa$ models the toughness of adhesion. The integrand $P$ represents the power of the prescribed displacement, see Definition~\ref{defi:SES} for its expression. Irreversibility of the debonding process is finally encoded in the solution directly asking that $t\mapsto A(t)$ has to be nondecreasing with respect to inclusion.

Such formulation of quasistatic brittle debonding has been first studied in \cite{BucButtLux}, where existence of a weaker notion of solution in terms of capacitary measures (rather than sets) have been proved, and then in \cite{MagRivTol}, where the authors showed existence of actual shape solutions to \eqref{eq:ensol} by resorting to an equivalent free-boundary reformulation which we will exploit also in this paper. See also \cite{RivQuas} for the simpler one-dimensional case.

In the cohesive setting the process is instead depicted in terms of displacements: the elastic energy and the dissipation (pseudo-)potential now read as $\frac 12 \int_\Omega|\nabla u|^2\d x$ and $\int_\Omega\Phi(x,u,\gamma)\d x$, respectively. Here, $\Omega\subseteq \R^d$ denotes the reference configuration of the membrane, while the cohesive density $(x,y,z)\mapsto\Phi(x,y,z)$, which somehow represents a regularized toughness, is nondecreasing in $y$ and $z$, and satisfies $\Phi(x,0,0)=0$ and $\lim\limits_{(y,z)\to \infty}\Phi(x,y,z)=\kappa(x)$. Solving the quasistatic cohesive debonding model means finding a map $t\mapsto (u(t),\gamma(t))$, where $\gamma$ is nondecreasing, fulfilling $\gamma(t)\ge u(t)$ and
\begin{equation}\label{eq:ensolcoh}
	\begin{cases}
		\displaystyle	\frac 12 \int_\Omega\!\!|\nabla u(t)|^2\d x+\!\int_\Omega\!\!\Phi(x,u(t),\gamma(t))\d x\le \frac 12 \int_\Omega|\nabla v|^2\d x+\!\int_\Omega\!\!\Phi(x,v,\gamma(t))\d x,\text{ for all suitable }v,\\
		\displaystyle\frac 12 \int_\Omega|\nabla u(t)|^2\d x+\int_\Omega\Phi(x,u(t),\gamma(t))\d x=\text{initial energy}+\int_0^t P(\tau)\,d\tau.
	\end{cases}
\end{equation}

The main difference with respect to the brittle scenario is given by the additional variable $\gamma$, which is crucial to preserve irreversibility; usually, $\gamma(t)$ describes the maximal opening occuring till time $t$, see \eqref{eq:gamma}. We point out that irreversibility cannot be asked directly on the displacement $u$, since it is not the movement of the membrane which is monotone, but the progression of degradation in the adhesion law. Thus, the additional variable $\gamma$ is necessary for a proper description of cohesive models. We also refer to \cite{CagnToad,DMZan,NegSca} for more details.

We are now in a position to describe the result of the paper. We consider the rescaled cohesive density $\Phi^\eps(x,y,z):=\Phi\left(x,\frac y\eps,\frac z\eps\right)$ and we pick any energetic solution $(u^\eps,\gamma^\eps)$ of the rescaled cohesive formulation \eqref{eq:ensolcoh}. As $\eps\to 0$ we prove that the sequence $u^\eps(t)$ converges to a limit displacement $u(t)$, whose \lq\lq irreversible\rq\rq positivity set
\begin{equation}\label{eq:Atilde}
	\widetilde A(t):=\bigcup_{s\in [0,t]}\{u(s)>0\}
\end{equation}
defines a solution to the brittle debonding model \eqref{eq:ensol}. We also mention that this cohesive-to-brittle approximation of debonding models can be seen as the evolutive counterpart of the singular perturbation analysis of the one-phase Bernoulli free-boundary problem, in the spirit of \cite[Chapter 1]{CaffSalsabook}.

In order to properly understand the meaning of this result and the difficulties arising in its proof, some considerations are due. First, the formal computation
\begin{equation*}
	\lim\limits_{\eps\to 0}\left(\frac 12\int_\Omega|\nabla u|^2\d x+\int_\Omega \Phi\left(x,\frac u\eps,\frac u\eps\right) \d x\right)=\underbrace{\frac 12\int_{\{u>0\}}|\nabla u|^2\d x}_{\approx\mc E(\{u>0\})}+ \int_{\{u>0\}}\kappa(x)\d x
\end{equation*} 
suggests that expression \eqref{eq:Atilde} is the natural candidate, encompassing irreversibility, for a quasistatic brittle solution. This can be actually made rigorous employing the equivalent reformulation of the brittle model in terms of displacements developed in \cite{MagRivTol}, which we will heavily exploit in our argument, see Proposition~\ref{prop:equiv}.

Second, in order to pass to the limit the rescaled version of \eqref{eq:ensolcoh}, especially the term $\displaystyle\int_\Omega \Phi\left(x,\frac{u^\eps(t)}{\eps},\frac{\gamma^\eps(t)}{\eps}\right) \d x$, besides compactness of $u^\eps$ (actually trivial due to obvious $H^1$ bounds) it is crucial to identify a suitable sequence of sets, somehow modelling an $\eps$-debonded region, encoding the irreversibility information. Since no perimeter bounds are actually available, just weak convergence of the characteristic functions of such sets can be obtained. This creates two issues: in principle the limit evolution may not even be described by sets, and even if it was, a link with the limit displacement $u$ (and so with \eqref{eq:Atilde}) is definitely not clear. We overcome these problems by introducing an artificial evolution of sets $A(t)$, see \eqref{eq:A(t)}, such that the pair $(u(t),A(t))$ fulfils a version of the global stability condition and of an energy inequality (obtained by sending $\eps\to 0$ in \eqref{eq:ensolcoh}) close to the equivalent reformulation of \cite{MagRivTol}. By means of some further inspections, this is enough to ensure that \eqref{eq:Atilde} is actually a solution to the brittle debonding model.

	The contribution of this paper is thus twofold. On the one hand, we validate the cohesive approximation of brittle models in the framework of quasistatic debonding. On the other hand, we provide an alternative proof of existence of energetic solutions to the brittle debonding model.
	
	The paper is organized as follows. In Section~\ref{sec:setting} we present in detail the framework of the present work and we list the needed assumptions. We describe both the brittle and the cohesive debonding model, and we state our main result, regarding the asymptotic analysis of rescaled cohesive evolutions, in Theorem~\ref{thm:main}. Section~\ref{sec:proof} is devoted to its proof. We first recall the equivalent free-boundary reformulation of brittle debonding introduced in \cite{MagRivTol}, and then we provide suitable compactness properties for both displacements and $\eps$-debonded sets which will lead to the convergence of rescaled cohesive evolutions to solutions of such reformulation. At the end of the paper we attach Appendix~\ref{app}, where we sketch the proof of existence of energetic solutions to the cohesive debonding model.

	\subsection*{Notations}
	
	The maximum (resp. minimum) of two extended real numbers $\alpha,\beta\in \R\cup\{\pm\infty\}$ is denoted by $\alpha\vee\beta$ (resp. $\alpha\wedge\beta$). For the positive and negative part of a real function $f$ we write $f^+:=f\vee 0$ and $f^-:=-(f\wedge 0)$, respectively. The standard scalar product between vectors $v,w\in \R^d$ is indicated by $v\cdot w$.
	
	Given a set $\Omega$ in $\R^d$, we denote by $\mathcal M(\Omega)$ the class of its Lebesgue measurable subsets, and by $L^0(\Omega)$ the space of Lebesgue measurable functions on $\Omega$. We adopt standard notations for Lebesgue, Sobolev and Bochner spaces. The space of absolutely continuous functions from an interval $[a,b]$ to a Banach space $X$ is denoted by $AC([a,b];X)$. For any family of scalar functions $Y(\Omega)$, we indicate its subset of non-negative elements by $Y(\Omega)^+$.
	
	When dealing with Lebesgue classes, we often avoid to explicitely write that a certain property holds almost everywhere (or a.e.), although it in fact does. For the same reason, through the paper inclusions or equalities of sets are to be meant up to sets of null Lebesgue measure.

	\section{Setting and main result}\label{sec:setting}
	
	Let $\Omega\subseteq \R^d$, $d\in \N$, be an open, bounded, connected Lipschitz set representing the reference configuration of an elastic adhesive membrane. We assume that the membrane is initially debonded, i.e. adhesion is not present, on an open subset $A_0$ of $\Omega$, namely in the paper we are not considering the problem of debond initiation. On $\Omega\setminus A_0$ instead, the adhesive effects will be differently described, depending on whether we consider the brittle or the cohesive regime.
	
	The debonding process is then triggered by an external, time-dependent, prescribed displacement $w$ acting on a portion $\Gamma$ of the unstuck boundary $\partial\Omega\cap \partial A_0$ with positive Hausdorff measure. For technical reasons, we require that 
	\begin{equation}\label{eq:neighborhood}
		\text{$A_0$ contains a neighborhood of $\Gamma$.}
	\end{equation}
	Although this assumption surely excludes certain situations, somehow related to debond initiation, it is completely natural in view of the mechanical applications, where some space is needed in order to apply the prescribed displacement. As customary, the considered prescribed displacement is the trace of a function defined on the whole of $\Omega$, and we assume that
	\begin{subequations}\label{hyp:w}
	\begin{equation}\label{eq:w}
		w\in AC([0,T];H^1(\Omega))\quad \text{ satisfies }\quad0\le w\le M\text{ in }[0,T]\times \Omega,
	\end{equation}
for a given positive constant $M>0$. Without loss of generality, by a standard mollification argument, due to \eqref{eq:neighborhood} we can suppose that 
\begin{equation}\label{eq:w0}
	w\equiv 0\qquad \text{in }[0,T]\times(\Omega\setminus A_0).
\end{equation}
\end{subequations}
	
	\subsection{Brittle debonding}
	
	In the brittle scenario, modelling an abrupt debonding process, the behaviour of the adhesion between the membrane and the substrate is described by a function $\kappa$, usually called toughness, which satisfies
	\begin{equation}\label{eq:kappa}
		\kappa\in L^\infty(\Omega)^+ \text{ such that }\kappa>0 \text{ on }\Omega\setminus A_0\text{ and }\kappa=0\text{ on }A_0.
	\end{equation}
The energy dissipated during the debonding process in order to move from a debonded configuration, described by a set $A$, to a larger one, described by a second set $B$, is then modeled by
\begin{equation}\label{eq:diss}
	\int_{B\setminus A}\kappa \d x.
\end{equation}

In the quasistatic framework, a robust and well-established variational notion of solution, which we will adopt in this paper, is provided by energetic solutions. We refer to \cite{MielkRoubbook} for an extensive presentation. Such solutions are based on two ingredients: a time-dependent energy driving the evolution, and a dissipation distance describing the energy loss in time. The resulting process is then governed by a global stability condition together with an energy(-dissipation) balance.

In our setting, the dissipation is described by \eqref{eq:diss}, augmented with a natural monotonicity constraint modelling irreversibility of the debonding phenomenon. In order to introduce the internal energy we borrow the notation from \cite{MagRivTol}, where quasistatic evolutions for brittle debonding have been analyzed. Given $\eta\in H^1(\Omega)$ and $A\in \mc M(\Omega)$ we set
\begin{align*}
	H^1_{\Gamma,\eta}(\Omega)&:=\{\varphi\in H^1(\Omega):\, \varphi=\eta \text{ on }\Gamma \},\\
	H^1_{\Gamma,\eta}(\Omega,A)&:=\{\varphi\in H^1_{\Gamma,\eta}(\Omega):\, \varphi=0 \text{ on }\Omega\setminus A\}.
\end{align*}
The energy of a measurable set $A$ at a time $t$ is then defined as
\begin{equation}\label{eq:E}
	\mc E(t,A):=\min\limits_{v\in H^1_{\Gamma,w(t)}(\Omega,A)}\frac 12 \int_\Omega|\nabla v|^2\d x,
\end{equation}
namely it consists in the minimial Dirichlet energy among displacements $v$ supported in $A$ and attaining the correct boundary condition $w(t)$ on $\Gamma$. We also denote by $\mathfrak{h}_{A,w(t)}$ the (unique) minimizer of \eqref{eq:E}, so that by definition
\begin{equation*}
	\mc E(t,A)=\frac 12 \int_\Omega|\nabla \mathfrak{h}_{A,w(t)}|^2\d x.
\end{equation*}

\begin{defn}\label{defi:SES}
	Under the previous assumptions, we say that a set-valued map $[0,T]\ni t\mapsto A(t)$ is a \emph{shape energetic solution} of the \emph{brittle} debonding model if the following conditions are satisfied (in order, compatibility, initial datum, irreversibility, global stability and energy balance):
	\begin{itemize}[leftmargin = !, labelwidth=1.2cm, align = left]
		\item [{\crtcrossreflabel{\textup{(CO)}$_\text{S}$}[def:COS]}] $A(t)$ is open for all $t\in [0,T]$;
		\item[{\crtcrossreflabel{\textup{(ID)}$_\text{S}$}[def:IDS]}] $A(0)=A_0$;
		\item[{\crtcrossreflabel{\textup{(IR)}$_\text{S}$}[def:IRS]}] $A(s)\subseteq A(t)$ for all $0\le s\le t\le T$;
		\item[{\crtcrossreflabel{\textup{(GS)}$_\text{S}$}[def:GSS]}] for all $t\in [0,T]$ there holds
		\begin{equation}\label{eq:GS}
			\mathcal E(t,A(t))\le \mathcal E(t,B)+\int_{B\setminus A(t)}\kappa\, dx,\quad \text{for all $B\in \mathcal M(\Omega)$ such that $A(t)\subseteq B$};
		\end{equation}
		\item[{\crtcrossreflabel{\textup{(EB)}$_\text{S}$}[def:EBS]}] the map $\displaystyle t\mapsto \int_\Omega \nabla \dot w(t)\cdot \nabla \mathfrak{h}_{A(t),w(t)} \, dx$ belongs to $L^1(0,T)$ and for all $t\in [0,T]$ there holds
		\begin{equation*}
			\mathcal E(t,A(t))+\int_{A(t)\setminus A_0}\kappa\,d x=\mc E(0,A_0)+\int_0^t  \int_\Omega \nabla \dot w(\tau)\cdot \nabla \mathfrak{h}_{A(\tau),w(\tau)} \, dx\, d\tau.
		\end{equation*}
	\end{itemize}
\end{defn}
	
	\subsection{Cohesive debonding}
	
	\begin{figure}
		\subfloat{\includegraphics[scale=.78]{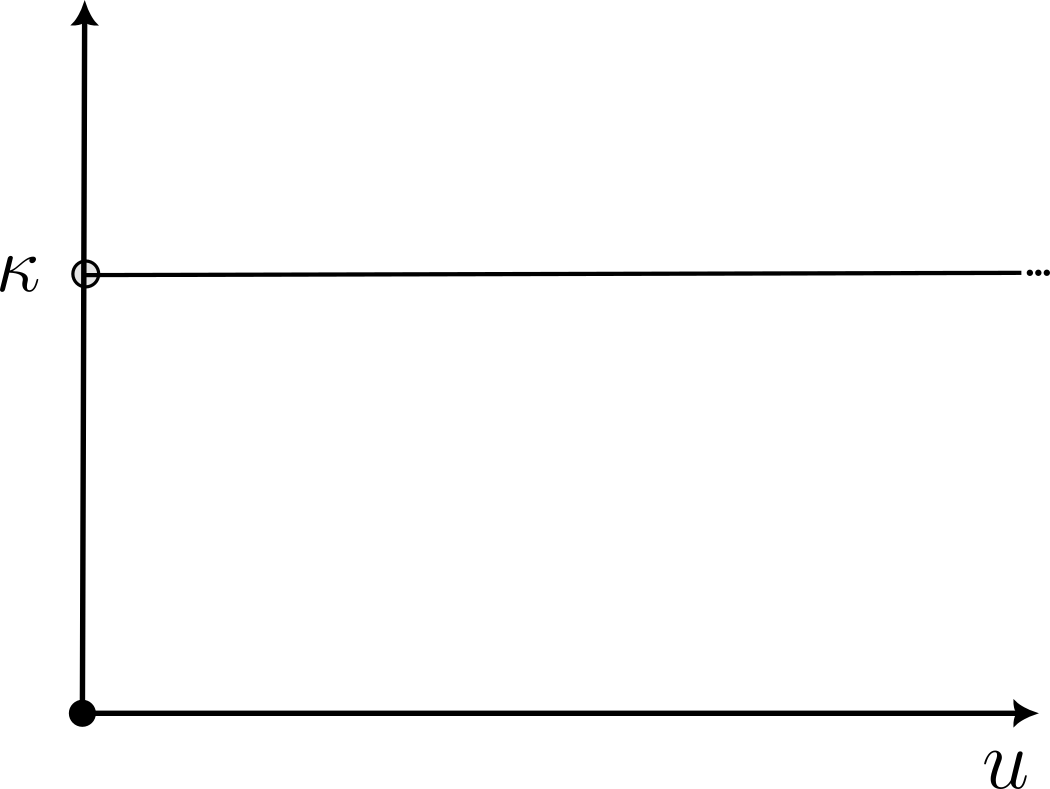}}\qquad\quad\subfloat{\includegraphics[scale=.78]{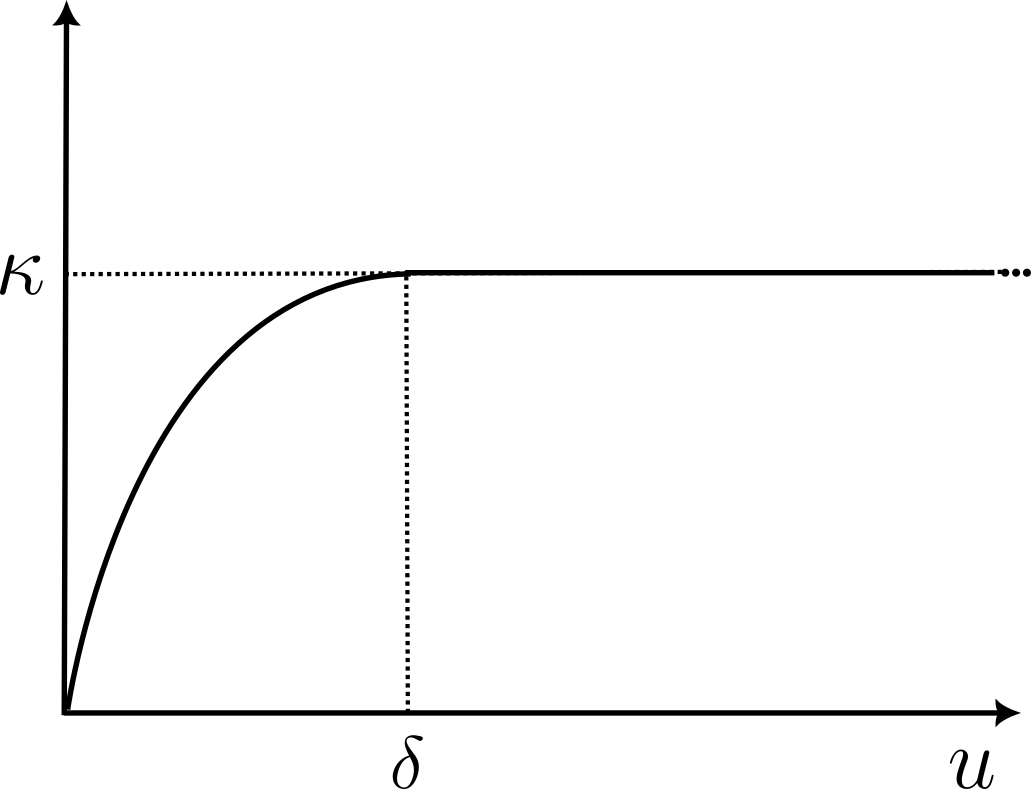}}\caption{Density of brittle adhesion (left) and cohesive adhesion (right) with respect to the displacement $u$.}\label{fig:1}
	\end{figure}
	
	While in brittle laws the adhesion is either completely active or completely broken, and in the latter case the (infinitesimal) dissipated energy equals the toughness $\kappa$, cohesive laws allow for intermediate states, somehow describing partial debonding. A proper description of this gradual behaviour is usually given in terms of displacements rather than of sets. Roughly speaking, as illustrated in Figure~\ref{fig:1}, complete debonding occurs just when the displacement overcomes a certain threshold $\delta>0$.
	
	Moreover, in order to incorporate the irreversibility of the debonding process in the model, an additional monotone variable (also called, history variable) $\gamma$, recording the maximal displacement reached during the evolution, is commonly introduced. In a certain sense, in the spirit of a damage variable, it is needed to trace the partial degradation the adhesion undergoes due to small elongations, whose possible interpretation is the rupture of a fraction of the microscopic adhesive bonds between membrane and substrate. In a smooth case, such history variable is simply given by
	\begin{equation}\label{eq:gamma}
		\gamma(t,x)=\sup\limits_{s\in [0,t]}|u(s,x)|,
	\end{equation}
where $u$ denotes the current displacement of the membrane.
	
	The cohesive energy we consider in the present paper is then modeled by
	\begin{equation}\label{eq:cohesive}
		\int_\Omega\Phi(x,|u|,\gamma)\d x,
	\end{equation}
	where the cohesive density $\Phi\colon \Omega\times [0,+\infty)\times[0,+\infty)\to [0,+\infty)$ is a Carath\'eodory function satisfying the following assumptions:
	
	\begin{enumerate}[label=\textup{($\Phi$\arabic*)}, start=1]	
		\item \label{hyp:phi1} $\Phi(x,0,0)=0$ for a.e. $x\in \Omega$;
		\item \label{hyp:phi2} $\Phi(x,y,z)\le K$ for some constant $K>0$, and for a.e. $x\in \Omega$, for all $y,z\ge 0$;
		\item \label{hyp:phi3} the map $y\mapsto \Phi(x,y,z)$ is Lipschitz continuous in $[0,+\infty)$, with Lipschitz constant independent of $x\in \Omega$ and of $z\ge 0$;
		\item \label{hyp:phi4} $\Phi(x,y,z)=\Phi(x,y,y{\vee}z)$ for a.e. $x\in\Omega$ and for all $y,z\ge 0$;
		\item \label{hyp:phi5} for a.e. $x\in\Omega$ the map $(y,z)\mapsto\Phi(x,y,z)$ is nondecreasing with respect to each component;
		\item\label{hyp:phi6} $\Phi(x,\cdot,\cdot)\equiv 0$ for a.e. $x\in A_0$, while $\Phi(x,y,z)> 0$ for a.e. $x\in \Omega\setminus A_0$ and for all $y,z> 0$.
	\end{enumerate}
Assumption \ref{hyp:phi6} is a compatibility condition between $\Phi$ and the initial debonded region, highlighting the fact that on $A_0$ no adhesion effects are present.

The dependence of $\Phi$ on the two variables $|u|$ and $\gamma$ allows to include in a proper and general way the different mechanical responses observed during two opposite regimes: the loading phase, when $|u|=\gamma$, and the unloading phase, when $|u|<\gamma$. The former encompasses the irreversible part of the cohesive process, while the latter describes residual elastic (and thus, reversible) effects, possibly due to the remaining active adhesive bonds. We refer to \cite{FredRiva, NegVit} for a deeper insight on the modelling of these two regimes.

\begin{ex}
		\begin{figure}
		\includegraphics[scale=.78]{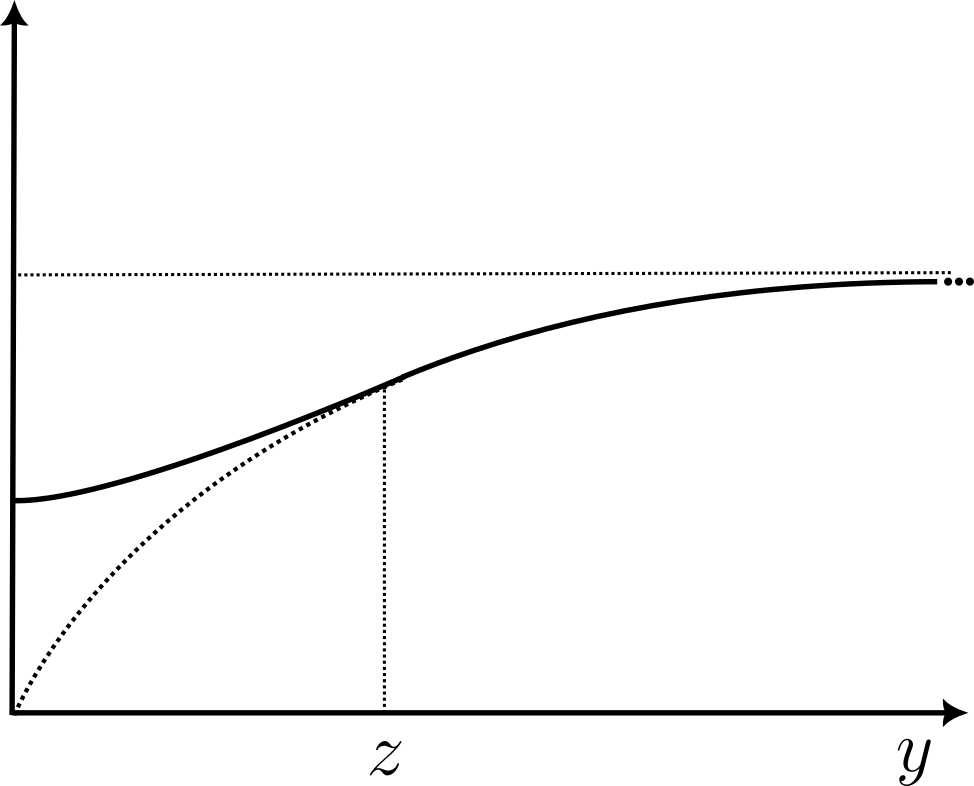}\caption{Graph of a typical cohesive density $y\mapsto \Phi(y,z)$, for a fixed $z>0$.}\label{fig:2}
	\end{figure}
	The typical example of cohesive energy density, featuring a quadratic unloading, is given by
	\begin{equation}\label{Phi}
		\Phi(y,z):=\begin{cases}\displaystyle
			\frac{\psi'(z)}{2z}y^2+\psi(z)-\frac{z\psi'(z)}{2}, &\text{if }y<z,\\
			\psi(y), &\text{otherwise},
		\end{cases}
	\end{equation}
	where for the sake of simplicity we are not considering the dependence on $x$, which anyway can be easily implemented in order to fulfil \ref{hyp:phi6}. The function $\psi\colon [0,+\infty)\to [0,+\infty)$, which models the loading phase, is nondecreasing, bounded, concave, of class $C^2$ and such that $\psi'(0)>0=\psi(0)$. See also Figure~\ref{fig:2}. It is not difficult to check that a density of the form \eqref{Phi} satisfies assumptions \ref{hyp:phi1}-\ref{hyp:phi5}. We refer to \cite[Proposition 1.4.]{Riv} for more details.
\end{ex}

We are now in a position to introduce the generalized notion of energetic solution we adopt in this paper to describe the quasistatic cohesive debonding model. The term generalized is here used since the history variable $\gamma$ is not required to fulfil exactly \eqref{eq:gamma}, but it is just a nondecreasing function larger than the maximal displacement. Nevertheless, we mention that equality \eqref{eq:gamma} could be recovered under stricter hypotheses, arguing as in \cite[Section 3.2]{Riv}. Anyway, in view of the asymptotic analysis illustrated in Section~\ref{sec:result}, our goal is to present the weakest notion of solution for which the convergence to a brittle evolution can be actually proved.

We finally point out that, differently from the brittle case where \eqref{eq:diss} describes all the occuring dissipation, expression \eqref{eq:cohesive} mixes dissipative effects (related to $\gamma$) with some residual internal energy (related to $|u|$). This contrast is reflected in the global stability condition \ref{def:GS} below, which presents a slightly different structure than \eqref{eq:GS}.

As an initial state we consider a pair $(u_0,\gamma_0)\in H^1_{\Gamma,w(0)}(\Omega)\times L^0(\Omega)^+$, and we assume it is globally stable, meaning that
\begin{equation}\label{eq:stable0}
	\begin{gathered}
		\gamma_0\ge |u_0|\quad\text{ and}\\
		\frac 12 \int_\Omega|\nabla u_0|^2\d x+\!\int_\Omega\!\Phi(x,|u_0|,\gamma_0)\d x\le \frac 12 \int_\Omega|\nabla v|^2\d x+\int_\Omega\!\Phi(x,|v|,\gamma_0)\d x, \text{ for all }v\in H^1_{\Gamma, w(0)}(\Omega).
	\end{gathered}
\end{equation}

\begin{defn}\label{def:genensol}
	Under the previous assumptions, we say that a map  $[0,T]\ni t\mapsto (u(t),\gamma(t))\in  H^1(\Omega)\times L^0(\Omega)^+$ is a \emph{generalized energetic solution} of the \emph{cohesive} debonding model if 
	
	\begin{itemize}[leftmargin = !, labelwidth=1.2cm, align = left]
		\item [{\crtcrossreflabel{\textup{(CO)}}[def:CO]}] $u(t)\in H^1_{\Gamma,w(t)}(\Omega)$ for all $t\in [0,T]$;
		\item[{\crtcrossreflabel{\textup{(ID)}}[def:ID]}] $(u(0),\gamma(0))=(u_0,\gamma_0)$;
		\item[{\crtcrossreflabel{\textup{(IR)}}[def:IR]}] $\gamma(s)\le \gamma(t)$ for all $0\le s\le t\le T$;
		\item[{\crtcrossreflabel{\textup{(GS)}}[def:GS]}] for all $t\in [0,T]$ there hold $\gamma(t)\ge |u(t)|$ and
		\begin{equation*}
			\frac 12 \int_\Omega|\nabla u(t)|^2\d x+\int_\Omega\Phi(x,|u(t)|,\gamma(t))\d x\le \frac 12 \int_\Omega|\nabla v|^2\d x+\int_\Omega\Phi(x,|v|,\gamma(t))\d x,
		\end{equation*}
	for all $v\in H^1_{\Gamma, w(t)}(\Omega)$;
		\item[{\crtcrossreflabel{\textup{(EB)}}[def:EB]}] the map $\displaystyle t\mapsto \int_\Omega \nabla \dot w(t)\cdot \nabla u(t) \, dx$ belongs to $L^1(0,T)$ and for all $t\in [0,T]$ there holds
		\begin{equation*}
			\begin{aligned}
				&\frac 12 \int_\Omega|\nabla u(t)|^2\d x+\int_\Omega\Phi(x,|u(t)|,\gamma(t))\d x\\
				=&\frac 12 \int_\Omega |\nabla u_0|^2\d x+\int_\Omega \Phi(x,|u_0|,\gamma_0) \d x+\int_0^t  \int_\Omega \nabla \dot w(\tau)\cdot \nabla u(\tau) \d x\d\tau.
			\end{aligned}
		\end{equation*}
	\end{itemize}
\end{defn}

Existence of generalized energetic solutions can be proved arguing as in \cite{BonCavFredRiva,FredRiva,Riv}, with small adjustments. For the sake of completeness, we provide a short proof in Appendix~\ref{app}.

The following proposition collects some basic properties of generalized energetic solutions.

\begin{prop}\label{prop:propertiesGES}
	Let $(u,\gamma)$ be a generalized energetic solution of the cohesive debonding model. Then $u(t)$ is nonnegative and bounded for all $t\in [0,T]$, and it holds
	\begin{equation}\label{eq:bounds}
		\sup\limits_{t\in [0,T]}\|u(t)\|_{L^\infty(\Omega)}\le M,\qquad \sup\limits_{t\in [0,T]}\|u(t)\|_{H^1(\Omega)}\le C,
	\end{equation}
where $M$ is the constant appearing in \eqref{eq:w}, while $C$ just depends on $\|w\|_{C^0([0,T];H^1(\Omega))}$ and on the constant $K$ from \ref{hyp:phi2}.

In particular, the map $\displaystyle t\mapsto \int_\Omega \nabla \dot w(t)\cdot \nabla u(t) \, dx$ is automatically in $L^1(0,T)$. 
\end{prop}

\begin{proof} 
	The $H^1$ bound easily follows by using $w(t)$ as a competitor in the global stability condition \ref{def:GS}. Indeed from \ref{hyp:phi2} we deduce
	\begin{equation*}
		\frac 12 \int_\Omega|\nabla u(t)|^2\d x\le \frac 12 \int_\Omega|\nabla w(t)|^2\d x+K|\Omega|,
	\end{equation*}
and we conclude by means of Poincar\'e inequality.
	
	To obtain the estimate $0\le u(t)\le M$ instead, we use as competitors the functions $u(t)^+$ and $u(t)\wedge M$, respectively. In the first case we infer
	\begin{align*}
		\frac 12 \int_\Omega|\nabla u(t)|^2\d x+\int_\Omega\Phi(x,|u(t)|,\gamma(t))\d x&\le \frac 12 \int_{\{u(t)>0\}}|\nabla u(t)|^2\d x+\int_\Omega\Phi(x,u(t)^+,\gamma(t))\d x\\
		&\le \frac 12 \int_{\{u(t)>0\}}|\nabla u(t)|^2\d x+\int_\Omega\Phi(x,|u(t)|,\gamma(t))\d x,
	\end{align*}
where the second inequality follows from \ref{hyp:phi5}. The above estimate yields $$\int_{\{u(t)\le 0\}}|\nabla u(t)|^2\d x=0,$$ whence $u(t)$ is nonnegative.

The very same argument with $u(t)\wedge M$ implies that $u(t)\le M$, and we conclude.
\end{proof}

	\subsection{From cohesive to brittle}\label{sec:result}
	
	We can now introduce the asymptotic problem at the core of the present work. To this aim, in order to link the cohesive and the brittle scenarios we additionally assume that the cohesive density $\Phi$ satisfies:
	\begin{enumerate}[label=\textup{($\Phi$\arabic*)}, start=7]
		\item\label{hyp:phi7} for a.e. $x\in \Omega$ there exists $\kappa_\Phi(x):=\lim\limits_{(y,z)\to \infty}\Phi(x,y,z)= \sup\limits_{y,z\ge 0}\Phi(x,y,z)$.
	\end{enumerate}
Such assumption introduces the proper notion of brittle toughness which will arise from the cohesive density. Note that we are requiring that $\kappa_\Phi$ is reached when the pair $(y,z)$ goes to infinity, and observe that the typical example \eqref{Phi} satisfies \ref{hyp:phi7} with $\kappa_\Phi=\lim\limits_{z\to \infty}\psi(z)$. Furthermore, due to \ref{hyp:phi2} and \ref{hyp:phi6}, the function $\kappa_\Phi$ directly fulfils \eqref{eq:kappa}.
	
For a small parameter $\eps>0$ we define the rescaled cohesive density 
	\begin{equation*}
		\Phi^\eps(x,y,z):=\Phi\left(x,\frac y\eps,\frac z\eps\right),
	\end{equation*}	
and we consider any generalized energetic solution $(u^\eps,\gamma^\eps)$ related to $\Phi^\eps$ and starting from an initial state $(u_0^\eps,\gamma_0^\eps)\in H^1_{\Gamma,w(0)}(\Omega)\times L^0(\Omega)^+$ which is globally stable in the sense of \eqref{eq:stable0} (with $\Phi^\eps$). Namely, recalling Proposition~\ref{prop:propertiesGES}, it satisfies
\begin{itemize}[leftmargin = !, labelwidth=1.2cm, align = left]
	\item [{\crtcrossreflabel{\textup{(CO)$^\eps$}}[def:COeps]}] $u^\eps(t)\in H^1_{\Gamma,w(t)}(\Omega)^+$ for all $t\in [0,T]$;
	\item[{\crtcrossreflabel{\textup{(ID)$^\eps$}}[def:IDeps]}] $(u^\eps(0),\gamma^\eps(0))=(u_0^\eps,\gamma_0^\eps)$;
	\item[{\crtcrossreflabel{\textup{(IR)$^\eps$}}[def:IReps]}] $\gamma^\eps(s)\le \gamma^\eps(t)$ for all $0\le s\le t\le T$;
	\item[{\crtcrossreflabel{\textup{(GS)$^\eps$}}[def:GSeps]}] for all $t\in [0,T]$ there hold $\gamma^\eps(t)\ge u^\eps(t)$ and
	\begin{equation*}
		\frac 12 \int_\Omega|\nabla u^\eps(t)|^2\d x+\int_\Omega\Phi\left(x,\frac{u^\eps(t)}{\eps},\frac{\gamma^\eps(t)}{\eps}\right)\d x\le \frac 12 \int_\Omega|\nabla v|^2\d x+\int_\Omega\Phi\left(x,\frac{|v|}{\eps},\frac{\gamma^\eps(t)}{\eps}\right)\d x,
	\end{equation*}
	for all $v\in H^1_{\Gamma, w(t)}(\Omega)$;
	\item[{\crtcrossreflabel{\textup{(EB)$^\eps$}}[def:EBeps]}] for all $t\in [0,T]$ there holds
	\begin{equation*}
		\begin{aligned}
			&\frac 12 \int_\Omega|\nabla u^\eps(t)|^2\d x+\int_\Omega\Phi\left(x,\frac{u^\eps(t)}{\eps},\frac{\gamma^\eps(t)}{\eps}\right)\d x\\
			=&\frac 12 \int_\Omega |\nabla u_0^\eps|^2\d x+\int_\Omega \Phi\left(x,\frac{u_0^\eps}{\eps},\frac{\gamma_0^\eps}{\eps}\right) \d x+\int_0^t  \int_\Omega \nabla \dot w(\tau)\cdot \nabla u^\eps(\tau)  \d x\d\tau.
		\end{aligned}
	\end{equation*}
\end{itemize}

The asymptotic analysis as $\eps\to 0$ of the cohesive rescaled evolution $(u^\eps,\gamma^\eps)$ to a brittle one is the content of our main result, Theorem~\ref{thm:main}. As often happens with singular limits, we require that the initial data $(u_0^\eps,\gamma_0^\eps)$ are well-prepared, in the sense that
\begin{equation}\label{eq:wellprepared}
	\lim\limits_{\eps\to 0}\int_\Omega \Phi\left(x,\frac{u_0^\eps}{\eps},\frac{\gamma_0^\eps}{\eps}\right) \d x=0.
\end{equation}
This ensures that no artificial debonding is produced by the initial state. Observe that the simplest situation in which both the global stability condition \eqref{eq:stable0} and \eqref{eq:wellprepared} are satisfied occurs when the membrane is initially at rest, i.e. $(u_0^\eps,\gamma_0^\eps)=(0,0)$. Indeed, in this case one has
\begin{equation*}
	\frac 12 \int_\Omega |\nabla u_0^\eps|^2\d x+\int_\Omega \Phi\left(x,\frac{u_0^\eps}{\eps},\frac{\gamma_0^\eps}{\eps}\right) \d x=\int_\Omega \Phi(x,0,0) \d x=0,
\end{equation*}
due to \ref{hyp:phi1}. Of course, this may happen just if $w(0)\equiv 0$ on $\Gamma$, which nevertheless represents a realistic and natural situation.
\begin{thm}\label{thm:main}
	Assume that the prescribed displacement $w$ satisfies \eqref{hyp:w} and the cohesive density $\Phi$ fulfils \ref{hyp:phi1}-\ref{hyp:phi7}. Let $(u^\eps,\gamma^\eps)$ be any generalized energetic solution of the rescaled cohesive debonding model starting from an initial state $(u_0^\eps,\gamma_0^\eps)$ satisfying \eqref{eq:stable0} and \eqref{eq:wellprepared}. Then there exists a map $u\colon [0,T]\to H^1(\Omega)^+$ such that, up to a nonrelabelled subsequence, for all $t\in [0,T]$ there holds
	\begin{equation*}
		u^\eps(t)\xrightarrow[\eps\to 0]{} u(t),\text{ strongly in $H^1(\Omega)$ and weakly$^*$ in }L^\infty(\Omega).
	\end{equation*}
Moreover, introducing the set
\begin{equation}\label{eq:Aut}
	A_{u(t)}:=A_0\cup\bigcup_{s\in [0,t]}\{u(s)>0\},
\end{equation}
it defines a shape energetic solution of the brittle debonding model in the sense of Definition~\ref{defi:SES}, with respect to the toughness $\kappa_\Phi$ introduced in \ref{hyp:phi7}.
\end{thm}

This theorem may look bizarre at a first sight, since it does not involve the history variable $\gamma^\eps$, which is the only source of irreversibility in the cohesive problem. Although it is not explicitely present in the statement, we however stress that its role will be crucial in the proof of the result. Indeed, the construction of an approximated debonded region will be based on $\gamma^\eps$, see \eqref{eq:Aepst}.

	\section{Proof of Theorem \ref{thm:main}}\label{sec:proof}
	
	The argument we propose in order to show the convergence of rescaled cohesive evolutions to brittle ones strongly relies on an equivalent reformulation of shape energetic solutions in terms of displacements, recently introduced in \cite{MagRivTol} as a proper way to prove existence of quasistatic solutions to the brittle debonding model by means of Minimizing Movements. 
	
	After recalling what we need of such reformulation in Section~\ref{sec:SES}, we present a suitable compactness result in Section~\ref{sec:compactness}. Besides displacements, whose compactness is actually trivial due to \ref{def:GSeps}, it will be crucial to keep track of a sort of cohesive debonded region, see \eqref{eq:Aepst}, in order to preserve irreversibility. Since no perimeter bounds are available for such sets, only weak compactness for their characteristic functions can be obtained by Helly's Selection Theorem, exploiting irreversibility. This allows us to first define an artificial limit debonded region, see \eqref{eq:A(t)}, and then to provide a connection between this limit set and the limit displacement, which are a priori unrelated. This is done in Proposition~\ref{prop:convu}. Section~\ref{sec:convergence} concludes the proof of our result by passing to the limit both \ref{def:GSeps} and \ref{def:EBeps}, finally showing that the limit displacement fulfils the equivalent reformulation previously introduced.
	
	\subsection{Preliminaries on shape energetic solutions}\label{sec:SES}
	
	We now present the characterization of shape energetic solutions via displacements, in the spirit of a free-boundary problem of Alt-Caffarelli type \cite{AltCaff}, developed in \cite{MagRivTol}. The underlying idea is that, due to maximum principle, the debonded region $A(t)$ should somehow correspond to the positivity set of the minimal displacement attaining the energy $\mc E(t,A(t))$. Taking into account irreversibility, this also justifies expression \eqref{eq:Aut}. 
	
	Next proposition, containing the aformentioned characterization, is obtained collecting some results from \cite{MagRivTol}, see in particular Propositions~2.5 and 3.3 and Remark 2.3 (recall that condition (2.10) therein is in force in our setting due to \eqref{eq:neighborhood}).
	
	\begin{prop}\label{prop:equiv}
		If a map $t\mapsto u(t)$ satisfies the following conditions, then the set $A_{u(t)}$ introduced in \eqref{eq:Aut} defines a shape energetic solution of the brittle debonding model.
		
		\begin{itemize}[leftmargin = !, labelwidth=1.2cm, align = left]
			\item [{\crtcrossreflabel{\textup{(CO)'}}[def:CO']}] $u(t)\in H^1_{\Gamma,w(t)}(\Omega)$ for all $t\in [0,T]$;
			\item[{\crtcrossreflabel{\textup{(ID)'}}[def:ID']}] $u(0)=\mathfrak{h}_{A_0,w(0)}$;
			\item[{\crtcrossreflabel{\textup{(GS)'}}[def:GS']}] for all $t\in [0,T]$ there holds
			\begin{equation*}
				\frac 12 \int_\Omega|\nabla u(t)|^2\d x\le \frac 12 \int_\Omega|\nabla v|^2\d x+\int_{\{v>0\}\setminus A_{u(t)}}\kappa\d x,\quad\text{for all $v\in H^1_{\Gamma, w(t)}(\Omega)$;}
			\end{equation*}
			\item[{\crtcrossreflabel{\textup{(EI)'}}[def:EI']}]  for all $t\in [0,T]$ there holds
			\begin{equation*}
				\frac 12 \int_\Omega|\nabla u(t)|^2\d x+\int_{ A_{u(t)}\setminus A_0}\kappa\d x\le\frac 12 \int_\Omega|\nabla u(0)|^2\d x+\int_0^t  \int_\Omega \nabla \dot w(\tau)\cdot \nabla u(\tau) \d x\d\tau.
			\end{equation*}
		\end{itemize}
	\end{prop}
\begin{rmk}
	Observe that conditions \ref{def:GS'} and \ref{def:EI'} involve the set $A_{u(t)}$ itself. This encompasses irreversibility and makes this formulation non-trivial.
\end{rmk}

	We also borrow the following technical lemma, which will be crucial for the asymptotic analysis, in particular for the proof of Proposition~\ref{prop:GS'}. We refer to \cite[Lemma 1.4]{MagRivTol} for a proof.
	
	\begin{lemma} \label{lemma:equivalent}
		Let $\eta\in H^1(\Omega)^+$ and let $A\in \mathcal M(\Omega)$. For $u \in H^1_{\Gamma,\eta}(\Omega,A)$ the following are equivalent:
		\begin{enumerate}			
			\item $\displaystyle \frac 12\int_\Omega |\nabla u|^2 \, dx  \leq \frac 12\int_\Omega |\nabla v|^2 \, dx + \int_{\{v > 0\}\setminus A}\kappa\, dx,\qquad$ for all $v\in H^1_{\Gamma,\eta}(\Omega)$;
			\item $\displaystyle \frac 12\int_\Omega |\nabla u|^2 \, dx+\int_{A}\kappa\, dx  \leq \frac 12\int_\Omega |\nabla v|^2 \, dx + \int_{\{v > 0\}}\kappa\, dx,\qquad$ for all $v\in H^1_{\Gamma,\eta}(\Omega)^+\cap C^0(\Omega)$ satisfying $A\subseteq \{v>0\}$.
		\end{enumerate}
		If one of the above holds, then $u=\mathfrak{h}_{A,\eta}$. Moreover, $u$ is locally Lipschitz continuous in $\Omega$ and it is harmonic in the open set $\{u>0\}$.
	\end{lemma}
	
	\subsection{Compactness}\label{sec:compactness}
	
	Although both considered notions of solution are now formulated in term of displacements, the cohesive history variable $\gamma^\eps$ is still the crucial ingredient in order to describe the debonded set. Indeed, it is the only element in the cohesive framework encompassing irreversibility. In order to clarify the expression of the $\eps$-debonded set $A^\eps(t)$ below, let us assume for a moment that $\Phi(x,\cdot,\cdot)$ is definitely constant, namely
	\begin{equation*}
		\Phi(x,y,z)=\kappa_\Phi(x),\quad\text{if }y\vee z\ge \delta(x),
	\end{equation*}
	for some positive threshold $\delta(x)>0$. Accordingly to the cohesive interpretation, the adhesion at a point $x$ is completely broken just when the displacement at that point has overcomed $\delta(x)$ at least once. Due to the rescaling, it is thus natural to describe the debonded region at time $t$ as the set $\{x\in \Omega:\, \gamma^\eps(t,x)\ge \eps\delta(x)\}$.
	
	In the general case, when $\kappa_\Phi$ may be just reached asymptotically, in order to mimic the previous consideration we pick an arbitrary sequence $\delta_\eps$ satisfying
	\begin{equation}\label{eq:deltaeps}
		\lim\limits_{\eps\to 0}\delta_\eps=+\infty,\qquad \lim\limits_{\eps\to 0}\eps\delta_\eps=0,
	\end{equation}
and for $t\in [0,T]$ we set
\begin{equation}\label{eq:Aepst}
	A^\eps(t):=A_0\cup\{\gamma^\eps(t)\ge \eps\delta_\eps\}.
\end{equation}

Observe that the map $t\mapsto A^\eps(t)$ is nondecreasing with respect to inclusion, due to \ref{def:IReps}, and that clearly $A_0\subseteq A^\eps(0)$. This is enough to infer weak compactness for the sequence of characteristic functions $\chi_{A^\eps(t)}$.

\begin{prop}\label{prop:convA}
	There exists a (non relabelled) subsequence and a map $t\mapsto \rho(t)$ such that for all $t\in [0,T]$ there holds
	\begin{equation}\label{eq:convrho}
		\chi_{A^\eps(t)}\xrightarrow[\eps\to 0]{}\rho(t)\qquad \text{weakly$^*$ in }L^\infty(\Omega).
	\end{equation}
In particular, $\rho$ is nondecreasing in time, $0\le\rho(t)\le 1$ for all $t\in [0,T]$ and $\rho(0)=1$ on $A_0$.
\end{prop}
\begin{proof}
	 The functions $\chi_{A^\eps(t)}$ are clearly bounded in $L^2(\Omega)$, as they take values in $\{0,1\}$. Moreover, by their time-monotonicity with respect to inclusion, they are also bounded in $BV([0,T];L^1(\Omega))$. Due to the reflexivity of $L^2(\Omega)$, by means of Helly's Selection Theorem together with a diagonal argument we conclude that, along a non-relabelled subsequence, there holds
	\begin{equation*}
		\chi_{A^{\eps}(t)} \xrightharpoonup[\eps\to 0]{} \rho(t) \quad \text{weakly in } L^2(\Omega) \quad \text{for all } \, t \in [0, T] \,.
	\end{equation*}

	Moreover, for any $t\in[0,T]$ the sequence $\chi_{A^\eps(t)}$ is also trivially bounded in $L^\infty(\Omega)$. Thus, from any subsequence we can extract a further subsequence that weakly$^*$ converges in $L^\infty(\Omega)$.
	Since the limit is uniquely identified and coincides with $ \rho(t) $, the desired convergence \eqref{eq:convrho} holds without a further extraction. 
	
	The stated properties on $\rho$ directly follow by weak$^*$ convergence from the fact that $\chi_{A^\eps(t)}$ is monotone and bounded between $0$ and $1$, and recalling that $A_0\subseteq A^\eps(0)$. 
\end{proof}

Although weak$^*$ convergence does not preserve characteristic functions, whence we do not directly have a limit debonded region, we can anyway identify an artificial one defining
\begin{equation}\label{eq:A(t)}
	A(t):=\{\rho(t)=1\}.
\end{equation}
Notice that $t\mapsto A(t)$ is nondecreasing with respect to inclusion, since $\rho$ is nondecreasing in time, and that $A_0\subseteq A(0)$.

Next proposition, besides an immediate compactness result for displacements, mainly provides the crucial link \eqref{eq:inclusion} between \eqref{eq:A(t)} and the limit displacement.

\begin{prop}\label{prop:convu}
	For all $t\in [0,T]$ there exist a further subsequence $\eps_n$ (possibly depending on $t$) and a function $u(t)\in H^1_{\Gamma,w(t)}(\Omega)^+\cap L^\infty(\Omega)$ such that
	\begin{equation}\label{eq:weakconv}
		u^{\eps_n}(t)\xrightarrow[n\to +\infty]{}u(t)\quad\text{ weakly in $H^1(\Omega)$ and weakly$^*$ in }L^\infty(\Omega).
	\end{equation}
In particular, $u(t)$ satisfies \ref{def:CO'}.

	Furthermore, for all $t\in [0,T]$ one has
	\begin{equation}\label{eq:inclusion}
		\{u(t)>0\}\subseteq A(t),
	\end{equation}
whence we deduce
\begin{equation}\label{eq:inclA}
	A_{u(t)}\subseteq A(t),\quad\text{ for all }t\in [0,T].
\end{equation}
\end{prop}
\begin{proof}
	The weak convergence stated in \eqref{eq:weakconv} immediately follows from the uniform bounds \eqref{eq:bounds}. Notice indeed that the involved constants are independent of $\eps$. In particular, $u(t)$ belongs to $ H^1_{\Gamma,w(t)}(\Omega)^+\cap L^\infty(\Omega)$. Moreover, by the compact embedding of $H^1(\Omega)$ into $L^2(\Omega)$, without loss of generality we can assume that $u^{\eps_n}(t)$ pointwise converges to $u(t)$ almost everywhere in $\Omega$.
	
	In order to prove \eqref{eq:inclusion}, we fix $t\in [0,T]$ and we pick $x\in \{u(t)>0\}$ such that $\lim\limits_{n\to +\infty} u^{\eps_n}(t,x)=u(t,x)$. Thus, recalling \eqref{eq:deltaeps}, there exists $\bar n=\bar n(t,x)\in \N$ such that for all $n\ge \bar n$ one has
	\begin{align*}
		u(t,x)&\ge 2\eps_n\delta_{\eps_n},\\
		u^{\eps_n}(t,x)-u(t,x)&\ge -\eps_n\delta_{\eps_n}.
	\end{align*} 
This yields
\begin{equation*}
	\gamma^{\eps_n}(t,x)\ge u^{\eps_n}(t,x)=u(t,x)+u^{\eps_n}(t,x)-u(t,x)\ge \eps_n\delta_{\eps_n},\quad\text{for }n\ge \bar n,
\end{equation*}
namely $x\in A^{\eps_n}(t)$ for $n\ge \bar n$. We have thus proved that $\lim\limits_{n\to +\infty}\chi_{A^{\eps_n}(t)}(x)=1$ for almost every $x\in \{u(t)>0\}$. Hence, by first exploiting \eqref{eq:convrho} and then using Dominated Convergence Theorem we obtain
\begin{equation*}
	\int_{\{u(t)>0\}}\rho(t)\d x=\lim\limits_{n\to +\infty}\int_{\{u(t)>0\}}\chi_{A^{\eps_n}(t)}(x)\d x= |\{u(t)>0\}|.
\end{equation*} 
Since $0\le \rho(t)\le 1$, the above equality implies that $\rho(t)=1$ almost everywhere in $\{u(t)>0\}$, and so we conclude by definition \eqref{eq:A(t)} of $A(t)$.
\end{proof}
	
	\subsection{Asymptotic analysis}\label{sec:convergence}
	
	In this final section we conclude the proof of Theorem~\ref{thm:main} by showing that the limit displacement obtained in Proposition~\ref{prop:convu} fulfils all the conditions stated in Proposition~\ref{prop:equiv}. This first proposition deals with the global stability condition.
	
	\begin{prop}\label{prop:GS'}
		The limit displacement $u(t)$ obtained in Proposition~\ref{prop:convu} satisfies \ref{def:GS'} and can be characterized by $u(t)=\mathfrak{h}_{A(t),w(t)}$ for all $t\in [0,T]$. In particular, the convergence stated in \eqref{eq:weakconv} holds for the whole subsequence found in Proposition~\ref{prop:convA}.
		
		Furthermore, there actually holds $A(0)=A_0$, whence \ref{def:ID'} is fulfilled.
	\end{prop}
\begin{proof}
	Fix $t\in [0,T]$ and pick $v\in H^1_{\Gamma,w(t)}(\Omega)^+\cap C^0(\Omega)$ such that $A(t)\subseteq \{v>0\}$, so that by \ref{def:GSeps} we have	
	\begin{equation}\label{eq:GSeps}
		\frac 12 \int_\Omega|\nabla u^{\eps_n}(t)|^2\d x+\int_\Omega\Phi\left(x,\frac{u^\epsn(t)}{\epsn},\frac{\gamma^\epsn(t)}{\epsn}\right)\d x\le \frac 12 \int_\Omega|\nabla v|^2\d x+\int_\Omega\Phi\left(x,\frac{v}{\epsn},\frac{\gamma^\epsn(t)}{\epsn}\right)\d x.
	\end{equation}
	Observe that the last term above can be bounded in this way:
	\begin{align*}
		&\int_\Omega\Phi\left(x,\frac{v}{\epsn},\frac{\gamma^\epsn(t)}{\epsn}\right)\d x\\
		=&\int_{\{v>0\}}\!\!\!\!\!\!\Phi\left(x,\frac{v}{\epsn},\frac{\gamma^\epsn(t)}{\epsn}\right)\d x+\int_{\{v=0\}\cap A^\epsn(t)}\!\!\!\!\!\!\Phi\left(x,0,\frac{\gamma^\epsn(t)}{\epsn}\right)\d x+\int_{\{v=0\}\setminus A^\epsn(t)}\!\!\!\!\!\!\Phi\left(x,0,\frac{\gamma^\epsn(t)}{\epsn}\right)\d x\\
		\le&\int_{\{v>0\}}\!\!\!\!\!\!\Phi\left(x,\frac{v}{\epsn},\frac{\gamma^\epsn(t)}{\epsn}\right)\d x+\int_{\{v=0\}\cap A^\epsn(t)}\!\!\!\!\!\!\Phi\left(x,0,\frac{\gamma^\epsn(t)}{\epsn}\right)\d x+\int_{\Omega\setminus A^\epsn(t)}\!\!\!\!\!\!\Phi\left(x,0,\frac{\gamma^\epsn(t)}{\epsn}\right)\d x.
	\end{align*}
On the other hand, by using \ref{hyp:phi5}, we can estimate from below the second term on the left-hand side of \eqref{eq:GSeps}:
\begin{align*}
	&\int_\Omega\Phi\left(x,\frac{u^\epsn(t)}{\epsn},\frac{\gamma^\epsn(t)}{\epsn}\right)\d x\\
	=&\int_{A^\epsn(t)}\Phi\left(x,\frac{u^\epsn(t)}{\epsn},\frac{\gamma^\epsn(t)}{\epsn}\right)\d x+\int_{\Omega\setminus A^\epsn(t)}\Phi\left(x,\frac{u^\epsn(t)}{\epsn},\frac{\gamma^\epsn(t)}{\epsn}\right)\d x\\
	\ge & \int_{A^\epsn(t)}\Phi\left(x,\frac{u^\epsn(t)}{\epsn},\frac{\gamma^\epsn(t)}{\epsn}\right)\d x+\int_{\Omega\setminus A^\epsn(t)}\Phi\left(x,0,\frac{\gamma^\epsn(t)}{\epsn}\right)\d x.
\end{align*}
By combining the previous two inequalities with \eqref{eq:GSeps} we end up with
\begin{equation}\label{eq:GSepsimproved}
	\begin{aligned}
		&\frac 12 \int_\Omega|\nabla u^{\eps_n}(t)|^2\d x+\int_{A^\epsn(t)}\Phi\left(x,\frac{u^\epsn(t)}{\epsn},\frac{\gamma^\epsn(t)}{\epsn}\right)\d x\\
		\le& \frac 12 \int_\Omega|\nabla v|^2\d x+\int_{\{v>0\}}\Phi\left(x,\frac{v}{\epsn},\frac{\gamma^\epsn(t)}{\epsn}\right)\d x+\int_{\{v=0\}\cap A^\epsn(t)}\Phi\left(x,0,\frac{\gamma^\epsn(t)}{\epsn}\right)\d x.
	\end{aligned}
\end{equation}

We now notice that
\begin{equation*}
\int_{A^\epsn(t)} \Phi\left(x,\frac{u^\epsn(t)}{\epsn},\delta_{\epsn}\right)\d x	\le\int_{A^\epsn(t)}\Phi\left(x,\frac{u^\epsn(t)}{\epsn},\frac{\gamma^\epsn(t)}{\epsn}\right)\d x\le \int_{A^\epsn(t)}\kappa_\Phi\d x.
\end{equation*}
The second inequality directly follows from \ref{hyp:phi7}, while the first one is a byproduct of \ref{hyp:phi5} together with the definition \eqref{eq:Aepst} of $A^\epsn(t)$ (recall that $\Phi(x,\cdot,\cdot)\equiv 0$ on $A_0$ by \ref{hyp:phi6}). The term on the right converges to $\int_\Omega\kappa_\Phi\rho(t)\d x$ as $n\to +\infty$ by \eqref{eq:convrho}. The term on the left does the same since the integrand strongly converges to $\kappa_\Phi$ in $L^1(\Omega)$ by Dominated Convergence Theorem, exploiting \ref{hyp:phi7} and \ref{hyp:phi2}.

We have thus proved that
\begin{subequations}
\begin{equation}\label{eq:1}
	\lim\limits_{n\to +\infty}\int_{A^\epsn(t)}\Phi\left(x,\frac{u^\epsn(t)}{\epsn},\frac{\gamma^\epsn(t)}{\epsn}\right)\d x=\int_\Omega \kappa_\Phi\rho(t)\d x.
\end{equation}
Analogously, one can show that
\begin{equation}\label{eq:2}
	\lim\limits_{n\to +\infty}\int_{\{v=0\}\cap A^\epsn(t)}\Phi\left(x,0,\frac{\gamma^\epsn(t)}{\epsn}\right)\d x=\int_{\{v=0\}} \kappa_\Phi\rho(t)\d x.
\end{equation}
On the other hand, since $v(x)/\epsn$ diverges as $n\to +\infty$ whenever $v(x)>0$, there also holds
\begin{equation}\label{eq:3}
	\lim\limits_{n\to +\infty}\int_{\{v>0\}}\Phi\left(x,\frac{v}{\epsn},\frac{\gamma^\epsn(t)}{\epsn}\right)\d x=\int_{\{v>0\}} \kappa_\Phi\d x.
\end{equation}
\end{subequations}

By sending $n\to +\infty$ in \eqref{eq:GSepsimproved}, exploiting weak lower semicontinuity of the first term and using \eqref{eq:1}, \eqref{eq:2}, \eqref{eq:3}, we infer
\begin{equation*}
	\frac 12\int_\Omega|\nabla u(t)|^2 \d x+\int_\Omega \kappa_\Phi\rho(t)\d x\le \frac 12 \int_\Omega|\nabla v|^2\d x+\int_{\{v>0\}} \kappa_\Phi\d x+\int_{\{v=0\}} \kappa_\Phi\rho(t)\d x.
\end{equation*}
Recalling that $\rho(t)=1$ on $A(t)$ by definition \eqref{eq:A(t)}, and that $\{v=0\}\subseteq\Omega\setminus A(t)$ since we chose $v$ satisfying $A(t)\subseteq \{v>0\}$, we finally obtain
\begin{equation*}
	\frac 12\int_\Omega|\nabla u(t)|^2 \d x+\int_{A(t)} \kappa_\Phi\d x\le \frac 12 \int_\Omega|\nabla v|^2\d x+\int_{\{v>0\}} \kappa_\Phi\d x.
\end{equation*}

Since $\{u(t)>0\}\subseteq A(t)$ by \eqref{eq:inclusion}, we are in a position to apply Lemma~\ref{lemma:equivalent}, which thus yields
\begin{equation*}
	\displaystyle \frac 12\int_\Omega |\nabla u(t)|^2 \, dx  \leq \frac 12\int_\Omega |\nabla v|^2 \, dx + \int_{\{v > 0\}\setminus A(t)}\kappa_\Phi\, dx,\qquad \text{for all } v\in H^1_{\Gamma,w(t)}(\Omega).
\end{equation*}
In turn, this inequality  implies \ref{def:GS'} by means of \eqref{eq:inclA}. Furthermore, Lemma~\ref{lemma:equivalent} also yields that $u(t)=\mf h_{A(t),w(t)}$. In particular, the convergence \eqref{eq:weakconv} does not depend on the chosen subsequence.

Finally, combining \eqref{eq:1} at time $t=0$ with assumption \eqref{eq:wellprepared} we deduce
\begin{equation*}
	0\ge \int_\Omega \kappa_\Phi\rho(0)\d x\ge \int_{A(0)\setminus A_0} \kappa_\Phi\d x,
\end{equation*}
and so $A(0)=A_0$ and we conclude.
\end{proof}

A similar argument  shows that the convergence of displacements is actually strong in $H^1(\Omega)$.

\begin{cor}
	Convergence \eqref{eq:weakconv} is strong in $H^1(\Omega)$.
\end{cor}
	\begin{proof}
	
	Fix $t\in [0,T]$. By choosing $v=u(t)$ as a competitor in \ref{def:GSeps}, the same computations performed in Proposition~\ref{prop:GS'} lead to
	\begin{equation*}
		\limsup\limits_{\eps\to 0}\frac 12\int_\Omega|\nabla u^\eps(t)|^2 \d x+\int_\Omega \kappa_\Phi\rho(t)\d x\le \frac 12 \int_\Omega|\nabla u(t)|^2\d x+\int_{\{u(t)>0\}} \kappa_\Phi\d x+\int_{\{u(t)=0\}} \kappa_\Phi\rho(t)\d x.
	\end{equation*}
	Since $\rho(t)=1$ almost everywhere on $\{u(t)>0\}$ by \eqref{eq:inclusion}, we finally infer
	\begin{equation*}
		\limsup\limits_{\eps\to 0}\frac 12\int_\Omega|\nabla u^\eps(t)|^2 \d x\le \frac 12 \int_\Omega|\nabla u(t)|^2\d x,
	\end{equation*}
	and we conclude.
\end{proof}

	This last proposition shows how to pass to the limit in the energy balance \ref{def:EBeps}, thus concluding the proof of Theorem \ref{thm:main}.

\begin{prop}
	The limit displacement $u(t)$ obtained in Proposition~\ref{prop:convu} satisfies \ref{def:EI'}.
\end{prop}
\begin{proof}
	Fix $t\in [0,T]$. From \ref{def:EBeps} we know that
		\begin{align*}
		&\frac 12 \int_\Omega|\nabla u^\eps(t)|^2\d x+\int_\Omega\Phi\left(x,\frac{u^\eps(t)}{\eps},\frac{\gamma^\eps(t)}{\eps}\right)\d x\\
		=&\frac 12 \int_\Omega|\nabla u^\eps_0|^2\d x+\int_\Omega\Phi\left(x,\frac{u^\eps_0}{\eps},\frac{\gamma^\eps_0}{\eps}\right)\d x+\int_0^t  \int_\Omega \nabla \dot w(\tau)\cdot \nabla u^\eps(\tau)  \d x\d\tau.
	\end{align*}
	The right-hand side easily converges to $\displaystyle \frac 12 \int_\Omega|\nabla u(0)|^2\d x+\int_0^t  \int_\Omega \nabla \dot w(\tau)\cdot \nabla u(\tau)  \d x\d\tau$ as $\eps\to 0$, due to the strong convergence of the displacements and to \eqref{eq:wellprepared}. As regards the left-hand side, observing that $$\displaystyle\int_\Omega\Phi\left(x,\frac{u^\eps(t)}{\eps},\frac{\gamma^\eps(t)}{\eps}\right)\d x\ge \int_{A^\eps(t)}\Phi\left(x,\frac{u^\eps(t)}{\eps},\frac{\gamma^\eps(t)}{\eps}\right)\d x$$ and recalling \eqref{eq:1} we deduce
	\begin{align*}
		\liminf\limits_{\eps\to 0}\left(\frac 12 \int_\Omega|\nabla u^\eps(t)|^2\d x+\int_\Omega\Phi\left(x,\frac{u^\eps(t)}{\eps},\frac{\gamma^\eps(t)}{\eps}\right)\d x\right)
		&\ge  \frac 12 \int_\Omega|\nabla u(t)|^2\d x+\int_\Omega\kappa_\Phi\rho(t) \d x\\
		&\ge\frac 12 \int_\Omega|\nabla u(t)|^2\d x+\int_{A(t)}\kappa_\Phi \d x\\
		&\ge\frac 12 \int_\Omega|\nabla u(t)|^2\d x+\int_{A_{u(t)}\setminus A_0}\kappa_\Phi \d x.
	\end{align*}
In the second inequality we exploited definition \eqref{eq:A(t)}, while in the last one we used \eqref{eq:inclA}. This concludes the proof.
\end{proof}

	\appendix
	\section{Existence of cohesive generalized energetic solutions}\label{app}
	
	We here provide a brief proof of the following result, along the lines of \cite{BonCavFredRiva,Riv}.
	
	\begin{thm}		
		Assume that the prescribed displacement $w$ satisfies \eqref{hyp:w} and the cohesive density $\Phi$ fulfils \ref{hyp:phi2}-\ref{hyp:phi6}. Given an initial state $(u_0,\gamma_0)\in H^1_{\Gamma,w(0)}(\Omega)\times L^0(\Omega)^+$ satisfying \eqref{eq:stable0}, there exists a generalized energetic solution of the cohesive debonding model.
	\end{thm}

	We employ the well-known Minimizing Movements scheme. Let $0=t^n_0<t^n_1<\dots<t^n_n=T$ be a sequence of partitions of $[0,T]$ such that 
	\begin{equation}\label{eq:fine}
		\lim\limits_{n\to +\infty}\sup\limits_{k=1,\dots, n}(t^n_k-t^n_{k-1})=0.
	\end{equation}
Starting from $(u_0,\gamma_0)$, for $k=1,\dots, n$ we consider the following iterative minimization algorithm:
	\begin{equation}\label{eq:minmov}
		\begin{cases}\displaystyle
			u_k\in \argmin\limits_{v\in H^1_{\Gamma,w(t^n_k)}(\Omega)}\left\{\frac 12\int_\Omega|\nabla v|^2 \d x+\int_\Omega\Phi(x,|v|,\gamma_{k-1}) \d x\right\},\\
			\gamma_k:=\gamma_{k-1}\vee |u_k|.
		\end{cases}
	\end{equation} 

By arguing as in Proposition~\ref{prop:propertiesGES}, we know that $u_k$ is nonnegative and
\begin{equation}\label{eq:bounduk}
	\max\limits_{k=0,\dots, n}(\|u_k\|_{L^\infty(\Omega)}+\|u_k\|_{H^1(\Omega)})\le C.
\end{equation}
This yields a uniform $L^\infty$ bound to $\widetilde{\gamma}_k:=0\vee u_1\wedge\dots\vee u_k$. Notice that $\gamma_k=\gamma_0\vee \widetilde{\gamma}_k$.

By exploiting \ref{hyp:phi3} and computing the Euler-Lagrange equations of \eqref{eq:minmov}, we additionally infer
\begin{equation*}
	\max\limits_{k=0,\dots, n}\|\Delta u_k\|_{L^\infty(\Omega)}\le C.
\end{equation*}
By $L^p$ elliptic regularity and an application of Sobolev Embedding Theorem we thus deduce
\begin{equation}\label{eq:lipbounduk}
	\max\limits_{k=0,\dots, n}\|u_k\|_{C^{0,1}_{\rm loc}(\Omega)}\le C,
\end{equation}
where $C^{0,1}_{\rm loc}(\Omega)$ denotes the space of locally Lipschitz continuous functions in $\Omega$. Since Lipschitz bounds are preserved with the operation of maximum, we also obtain
\begin{equation}\label{eq:lipbound}
	\max\limits_{k=0,\dots, n}\|\widetilde{\gamma}_k\|_{C^{0,1}_{\rm loc}(\Omega)}\le C.
\end{equation}

We now introduce the piecewise constant interpolants
\begin{equation*}
	u^n(t):=\begin{cases}
		u_k,&\text{if }t\in [t^n_k,t^n_{k+1}),\\
		u_n,&\text{if }t=T,
	\end{cases}\qquad
\tau^n(t):=\begin{cases}
	t^n_k,&\text{if }t\in [t^n_k,t^n_{k+1}),\\
	T,&\text{if }t=T,
\end{cases}
\end{equation*}
and analogous expressions for $\widetilde{\gamma}^n(t)$ and $\gamma^n(t)$. Notice that $\gamma^n(t)=\gamma_0\vee \widetilde{\gamma}^n(t)$. We also set $w^n(t):=w(\tau^n(t))$.

Exploiting \eqref{eq:lipbound}, Helly's Selection Theorem together with Ascoli-Arzel\'a Theorem ensure the existence of a subsequence (not relabelled) such that for all $t\in [0,T]$ there holds
\begin{equation*}
	\widetilde\gamma^n(t)\xrightarrow[n\to +\infty]{}\widetilde{\gamma}(t),\quad\text{locally uniformly in }\Omega.
\end{equation*}
In particular, setting $\gamma(t):=\gamma_0\vee \widetilde{\gamma}(t)$, we have
\begin{equation*}
	\gamma^n(t)\xrightarrow[n\to +\infty]{}{\gamma}(t),\quad\text{in }L^\infty_{\rm loc}(\Omega).
\end{equation*}
Moreover, by \eqref{eq:bounduk} and \eqref{eq:lipbounduk}, for any $t\in [0,T]$ we can extract a further subsequence (possibly depending on $t$) such that
\begin{equation*}
	u^{n_j}(t)\xrightarrow[j\to +\infty]{}u(t),\quad\text{weakly in $H^1(\Omega)$ and locally uniformly in }\Omega.
\end{equation*}
Previous convergences are enough to grant the validity of \ref{def:CO}, \ref{def:ID} and \ref{def:IR} of Definition~\ref{def:genensol}; we are just left to show \ref{def:GS} and \ref{def:EB}.

As regards the former, we first observe that the condition $\gamma(t)\ge u(t)$ follows again by the previous convergences and the definition of $\gamma_k$. We then pick $v\in H^1_{\Gamma, w(t)}(\Omega)$, and by choosing $v+w^{n_j}(t)-w(t)$ as a competitor in \eqref{eq:minmov}, we deduce
\begin{align*}
	&\frac 12 \int_\Omega|\nabla u(t)|^2\d x+\int_\Omega\Phi(x,u(t),\gamma(t))\d x\\
	\le&\liminf\limits_{j\to +\infty}\left[\frac 12 \int_\Omega|\nabla u^{n_j}(t)|^2\d x+\int_\Omega\Phi(x,u^{n_j}(t),\gamma^{n_j}(t))\d x\right]\\
	=&	 \liminf\limits_{j\to +\infty}\left[\frac 12 \int_\Omega|\nabla u^{n_j}(t)|^2\d x+\int_\Omega\Phi(x,u^{n_j}(t),\gamma^{n_j}(t-\tau^{n_j}(t)))\d x\right]\\
	\le & \liminf\limits_{j\to +\infty}\left[\frac 12 \int_\Omega|\nabla v+\nabla(w^{n_j}(t)-w(t))|^2\d x+\int_\Omega\Phi(x,|v+w^{n_j}(t)-w(t)|,\gamma^{n_j}(t-\tau^{n_j}(t)))\d x\right],
\end{align*}
where the only equality above is due to \ref{hyp:phi4}. The first term in the last line above converges to $\frac 12 \int_\Omega |\nabla v|^2 \d x$, since $w^n(t)\xrightarrow[n\to+\infty]{}w(t)$ strongly in $H^1(\Omega)$, while to deal with the second one we argue as follows. By using first \ref{hyp:phi6} (and \eqref{eq:w0}) and then \ref{hyp:phi5} we observe that
\begin{equation*}
	\int_\Omega\!\!\Phi(x,|v+w^{n_j}(t)-w(t)|,\gamma^{n_j}(t-\tau^{n_j}(t)))\d x\!=\!\!\!\int_\Omega\!\!\Phi(x,|v|,\gamma^{n_j}(t-\tau^{n_j}(t)))\d x\le\!\! \int_\Omega\!\!\Phi(x,|v|,\gamma^{n_j}(t))\d x,
\end{equation*}
which converges to $\int_\Omega\Phi(x,|v|,\gamma(t))\d x$. Thus \ref{def:GS} is proved.

The energy balance \ref{def:EB} is finally a byproduct of the following two lemmas.

\begin{lemma}
	There exists a vanishing sequence $R^n$ such that for all $t\in [t^n_1,T]$ there holds
	\begin{equation*}
		\begin{aligned}
			&\frac 12 \int_\Omega|\nabla u^n(t)|^2\d x+\int_\Omega\Phi(x,u^n(t),\gamma^n(t))\d x\\
			\le&\frac 12 \int_\Omega |\nabla u_0|^2\d x+\int_\Omega \Phi(x,u_0,\gamma_0) \d x+\int_0^t  \int_\Omega \nabla \dot w(\tau)\cdot \nabla u^n(\tau) \d x\d\tau+R^n.
		\end{aligned}
	\end{equation*}
\end{lemma}
\begin{proof}
	Fix $t\in[t^n_1,T]$ and let $k$ such that $t\in[t^n_k,t^n_{k+1})$. For $j=1,\dots, k$, choosing $u_{j-1}+w(t^n_j)-w(t^n_{j-1})$ as competitor for $u_j$, and exploiting \ref{hyp:phi4} and \ref{hyp:phi6}, we deduce
	\begin{align*}
		&\frac 12\int_\Omega|\nabla u_j|^2 \d x+\int_\Omega\Phi(x,u_j,\gamma_{j}) \d x=\frac 12\int_\Omega|\nabla u_j|^2 \d x+\int_\Omega\Phi(x,u_j,\gamma_{j-1}) \d x\\
		\le&\frac 12\int_\Omega|\nabla u_{j-1}+\nabla( w(t^n_j)-w(t^n_{j-1}))|^2 \d x+\int_\Omega\Phi(x,|u_{j-1}+w(t^n_j)-w(t^n_{j-1})|,\gamma_{j-1}) \d x\\
		=&\frac 12\int_\Omega|\nabla u_{j-1}+\nabla( w(t^n_j)-w(t^n_{j-1}))|^2 \d x+\int_\Omega\Phi(x,u_{j-1},\gamma_{j-1}) \d x.
	\end{align*}
	Subtracting $\frac 12\int_\Omega|\nabla u_{j-1}|^2 \d x$ to both sides and summing from $j=1,\dots, k$ we infer
	\begin{align*}
		\frac 12\int_\Omega|\nabla u_k|^2 \d x+\int_\Omega\Phi(x,u_k,\gamma_{k})& \d x
		\le \frac 12 \int_\Omega |\nabla u_0|^2\d x+\int_\Omega \Phi(x,u_0,\gamma_0) \d x\\&+\sum_{j=1}^{k}\int_{t^n_{j-1}}^{t^n_j}\int_\Omega \nabla \dot w(\tau)\cdot (\nabla u_{j-1}+\nabla(w(\tau)-w(t^n_{j-1}))) \d x\d\tau.
	\end{align*}
Rewriting the above inequality in terms of $u^n(t)$ and $\gamma^n(t)$ we thus obtain
\begin{align*}
	&\frac 12 \int_\Omega|\nabla u^n(t)|^2\d x+\int_\Omega\Phi(x,u^n(t),\gamma^n(t))\d x\\
	\le&\frac 12 \int_\Omega |\nabla u_0|^2\d x+\int_\Omega \Phi(x,u_0,\gamma_0) \d x+\int_0^t  \int_\Omega \nabla \dot w(\tau)\cdot \nabla u^n(\tau) \d x\d\tau\\
	&+\underbrace{\int_0^{\tau^n(t)}\int_\Omega  \nabla \dot w(\tau)\cdot \nabla(w(\tau)-w^n(\tau)) \d x\d\tau-\int_{\tau^n(t)}^t\int_\Omega \nabla \dot w(\tau)\cdot\nabla u^n(\tau) \d x\d\tau}_{=:r^n(t)}.
\end{align*}
By exploiting \eqref{eq:w} and \eqref{eq:bounduk}, and recalling \eqref{eq:fine}, it is then standard to show that $|r^n(t)|$ can be bounded, uniformly in $t$, by a vanishing sequence $R^n$. So we conclude.
\end{proof}

\begin{lemma}
	Let the pair  $(u,\gamma)$ satisfy \ref{def:ID} and \ref{def:GS}. Then for all $t\in [0,T]$ there holds
	\begin{equation*}
		\begin{aligned}
			&\frac 12 \int_\Omega|\nabla u(t)|^2\d x+\int_\Omega\Phi(x,u(t),\gamma(t))\d x\\
			\ge&\frac 12 \int_\Omega |\nabla u_0|^2\d x+\int_\Omega \Phi(x,u_0,\gamma_0) \d x+\int_0^t  \int_\Omega \nabla \dot w(\tau)\cdot \nabla u(\tau) \d x\d\tau.
		\end{aligned}
	\end{equation*}
\end{lemma}
\begin{proof}
	Fix $t\in (0,T]$ and consider a sequence of partitions $0=s^n_0<\dots<s^n_n=t$ satisfying:
	\begin{enumerate}
		\item 	$\lim\limits_{n\to +\infty}\sup\limits_{k=1,\dots, n}(s^n_k-s^n_{k-1})=0;$
		\item $\displaystyle \lim\limits_{n\to +\infty}\sum_{k=1}^{n}(s^n_k-s^n_{k-1})\int_\Omega \nabla\dot w(s^n_k)\cdot\nabla u(s^n_k)\d x=\int_0^t\int_\Omega \nabla\dot w(\tau)\cdot\nabla u(\tau) \d x\d\tau$;
		\item $\displaystyle \lim\limits_{n\to +\infty}\sum_{k=1}^{n}\left\|(s^n_k-s^n_{k-1})\dot w(s^n_k)-\int_{s^n_{k-1}}^{s^n_k}\dot w(\tau)\d\tau\right\|_{H^1(\Omega)}=0$;
		\item $\displaystyle \sup\limits_{k=1,\dots, n}\int_{s^n_{k-1}}^{s^n_k}\|\dot w(\tau)\|_{H^1(\Omega)}\d\tau\le \frac 1n.$
	\end{enumerate}
Picking the function $u(s^n_k)+w(s^n_{k-1})-w(s^n_k)$ as a competitor for $u(s^n_{k-1})$ in \ref{def:GS} we infer
\begin{align*}
	&\frac 12 \int_\Omega|\nabla u(s^n_{k-1})|^2\d x+\int_\Omega\Phi(x,u(s^n_{k-1}),\gamma(s^n_{k-1}))\d x\\
	\le&\frac 12 \int_\Omega |\nabla u(s^n_k)+\nabla (w(s^n_{k-1})-w(s^n_k))|^2\d x+\int_\Omega \Phi(x,|u(s^n_k)+w(s^n_{k-1})-w(s^n_k)|,\gamma(s^n_{k-1})) \d x\\
	\le & \frac 12 \int_\Omega |\nabla u(s^n_k)+\nabla (w(s^n_{k-1})-w(s^n_k))|^2\d x+ \int_\Omega \Phi(x,u(s^n_k),\gamma(s^n_{k})) \d x,
\end{align*}
where in the last inequality we used \ref{hyp:phi5} and \ref{hyp:phi6}. By adding $\frac 12 \int_\Omega|\nabla u(s^n_{k})|^2\d x$ to both sides and by summing from $k=1$ to $k=n$ we deduce
\begin{align*}
	&\frac 12 \int_\Omega|\nabla u(t)|^2\d x+\int_\Omega\Phi(x,u(t),\gamma(t))\d x-\frac 12 \int_\Omega |\nabla u_0|^2\d x-\int_\Omega \Phi(x,u_0,\gamma_0) \d x\\
	\ge&\sum_{k=1}^n\int_{s^n_{k-1}}^{s^n_k}  \int_\Omega \nabla \dot  w(\tau)\cdot (\nabla u(s^n_k)+\nabla(w(\tau)-w(s^n_{k}))) \d x\d\tau.
\end{align*}
By exploiting (1)-(4) it is standard to show that last term above converges to $\int_0^t  \int_\Omega \nabla \dot w(\tau)\cdot \nabla u(\tau) \d x\d\tau$, and we conclude.
\end{proof}

%
%
%

\bigskip

\noindent\textbf{Acknowledgements.}
The author is member of GNAMPA (INdAM) and acknowledges its support through the INdAM-GNAMPA project 2025 \lq\lq DISCOVERIES\rq\rq (CUP E5324001950001).

\end{document}